\documentclass[preprint,11pt,authoryear]{elsarticle}

\usepackage{amssymb}
\usepackage{amsthm}

\newtheorem{proposition}{Proposition}
\newtheorem{lemma}{Lemma}
\newtheorem{example}{Example}

\newtheorem{result}{Result}
\usepackage{amsmath}
\DeclareMathOperator*{\argmin}{arg\,min}
\usepackage{mathtools}
\usepackage{comment}
\usepackage{xcolor}
\usepackage{csquotes}
\usepackage{natbib}
\usepackage{comment}
\usepackage{algorithm}
\usepackage{algpseudocode}
\usepackage{booktabs}
\usepackage{pgfplots}
\pgfplotsset{compat=newest}
\usepackage{caption}
\usepackage{subcaption}
\usepackage{pgfplotstable}
\usepackage{etoolbox}
\usepackage{adjustbox}
\usepackage{setspace}
\newif\ifinclude
 \includefalse 
\journal{Elsevier}

\begin{document}

\begin{frontmatter}



\title{Optimal taxes and subsidies to incentivize modal shift for inner-city freight transport}


\author[inst1] {K. Tundulyasaree}\corref{cor1}
\ead{k.tundulyasaree@tue.nl}

\affiliation[inst1]{organization={Department of Industrial Engineering \& Innovation Sciences, Eindhoven University of Technology},
            city={Eindhoven},
            country={the Netherlands}}

\author[inst1]{L. Martin}
\author[inst1]{R. N. van Lieshout}
\author[inst1]{T. Van Woensel}

\cortext[cor1]{Corresponding author}

\begin{abstract}
With increasing freight demands for inner-city transport, shifting freight from road to scheduled line services such as buses, metros, trams, and barges is a sustainable solution. Public authorities typically impose economic policies, including road taxes and subsidies for scheduled line services, to achieve this modal shift. This study models such a policy using a bi-level approach: at the upper level, authorities set road taxes and scheduled line subsidies, while at the lower level, freight forwarders arrange transportation via road or a combination of road and scheduled lines. We prove that fully subsidizing the scheduled line is an optimal and budget-efficient policy. Due to its computational complexity, we solve the problem heuristically using a bi-section algorithm for the upper level and an Adaptive Large Neighbourhood Search for the lower level. Our results show that optimally setting subsidy and tax can reduce the driving distance by up to 12.5\% and substantially increase modal shift, albeit at a higher operational cost due to increased taxes. Furthermore, increased scheduled line frequency and decreased geographical scatteredness of freight orders increase modal shift. For the partial subsidy policy, we found that an additional budget provides a better trade-off between minimizing distance and transportation costs than solely increasing the subsidy level. In a Berlin, Germany, case study, we find that we can achieve up to 2.9\% reduction in driven distance due to 23.2\% scheduled line usage, which amounts to an increase of multiple orders of magnitude, despite only using a few stations for transshipment.
\end{abstract}


\begin{highlights}
\item Recycled tax as train subsidy can increase freight modal shift
\item A full subsidy is provably an optimal policy
\item Optimal tax and subsidy reduce the driven distance by up to 12.5\%
\end{highlights}

\begin{keyword}
Transportation \sep Routing \sep Heuristics \sep Logistics 
\end{keyword}
\end{frontmatter}
\section{Introduction}

Integrating scheduled line services for inner-city freight transport can help address the challenges of increasing demand. E-commerce sales are projected to reach \$6.3 trillion with an 8.8\% growth in 2024 \citep{forbes2024ecommerce}. This surge in demand for urban transport contributes to air pollution, road congestion, and a decline in quality of life. A promising solution is integrating freight transport into passenger scheduled line services. These services utilize existing urban transportation operating on fixed timetables, such as buses, metros, trams, and barges. During off-peak periods, the spare capacity in these services can be utilized for freight, increasing system efficiency and reducing the impact of road freight vehicles by minimizing road distances \citep{savelsbergh201650th}. Nonetheless, implementing this modal shift is difficult because it involves extra time and handling costs for transshipment. Therefore, proper incentives are required to encourage freight forwarders to shift the freight off the road. 

Carbon taxes and subsidies for modal shift are common economic policies. \cite{takman2023public} reviewed more than 90 European projects for modal shift since 2000 and found that most of these projects target economic policies, including taxes and subsidies, to encourage a shift to railway transportation. According to ex-ante reports, national-level grants and subsidies have shown positive performance, while EU-level policies have had mixed success. Despite total EU funding of around €1.1 billion, the EU’s regulatory and financial support for intermodality has not been sufficient for intermodal freight transport to compete effectively with road transportation \citep{EuropeanCourt2023}. Tax policies often face societal opposition due to higher costs, necessitating reinvesting tax revenue \citep{jagers2019impact}. Therefore, the proper settings for these economic policies are crucial to ensure their success.

In this paper, we aim to investigate the effect of road tax and scheduled line service subsidy on increasing the freight modal shift for inner-city freight transport settings. To achieve this, we develop a bi-level model to determine the optimal road tax and scheduled line service subsidy, where the subsidy is derived from recycled tax revenue and a given budget. The upper level of the model represents the perspective of a transportation authority, which targets to minimize total road freight distance by imposing road taxes and setting the scheduled line service subsidy. The lower level represents the perspective of a freight forwarder who aims to minimize transportation costs within the imposed system. We analyze the theoretical properties of this model and derive the structure of optimal policies. In addition, we developed a combined adaptive large neighborhood search and a bisection method to solve artificially generated instances, as well as a case study based on the city of Berlin.

Our contributions are threefold. First, we make theoretical contributions by showing that the policy with a fully subsidized scheduled line achieves the minimum possible driving distance and deriving other properties under fixed and variable budget levels. We found that fully subsidized scheduled line services represent an optimal policy, achieving the best trade-off among stakeholders. Second, we show that the optimal policy can result in high modal shifts and significant driving distance savings, albeit at higher system operation costs. These costs may be mitigated with additional budget allocations from the authority. Lastly, we demonstrate the benefits of our approach with a real Berlin case study, showing around a 3\% reduction in driving distance.

The remainder of this paper is structured as follows: In Section \ref{sec:LiteratureReview}, we review related research on economic policies in freight transportation and the literature on operating and planning scheduled lines integrated into inner-city freight transport. In Section \ref{sec:ProblemStatement}, we describe the problem and introduce the bi-level formulation of the problem. In Section \ref{sec:PropertiesOptimal}, we provide optimal solution properties. Furthermore, Section \ref{sec:Algorithm} presents the combined Bisection and Adaptive Large Neighborhood Search algorithm for solving the proposed problem. Section \ref{sec:NumericalExp} outlines the experimental designs and discusses all the numerical experiments. Finally, Section \ref{sec:Conclusion} concludes the study.

\section{Literature Review}
\label{sec:LiteratureReview}

This paper relates to (i) research on modal shift in inner-city freight transport as well as (ii) studies on operational planning in inner-city freight transport, focusing on the integration of truck transport and scheduled line services. 

\subsection{Modal Shift Regulations}
Subsidies and taxes are key public policy tools for promoting modal shift in freight transport \citep{takman2023public}. They incentivize freight forwarders to adapt their decisions towards more coordinated and sustainable operations. Authorities typically set subsidies and taxes to force freight forwarders to internalize external transport costs, like emissions and congestion \citep{santos2010part}. These policies face two main challenges: political feasibility and economic effectiveness. Political acceptability often encounters societal opposition, but revenue recycling -- spending that benefits stakeholders -- can increase support \citep{carattini2018overcoming,jagers2019impact}. For example, \cite{beiser2019could} found that most US sampled group supports carbon taxes as long as the revenue is invested in new infrastructure, renewable energy, and policies for low-income families, and providing tax rebates. Moreover, effectiveness depends on the price levels needed to induce behavior change. When setting the price, it is crucial to consider the hierarchical relationships between stakeholders, i.e., the decision from a stakeholder has a direct impact/effect on another stakeholder's decisions. 

We can model tax and policy setting using a pricing problem where the leader sets prices for certain activities to maximize benefits. At the same time, the followers choose activities to minimize operating costs \citep{labbe2016bilevel}. This model represents taxes as positive prices and subsidies as negative prices. Several studies consider how to allocate a given budget as subsidies, such as intermodal subsidies \citep{hu2022optimal}, bus service contracts \citep{he2023promoting}, and rail freight subsidies \citep{yin2024low, mohri2022designing}. Other studies \citep{bruni2024bi, labbe1998bilevel, brotcorne2008joint} focus exclusively on the taxation of intermediate hub facilities, tollways, or telecommunication networks without addressing the allocation of the additional revenues generated. We refer to the comprehensive review by \cite{caselli2024bilevel} for recent work adopting this price-setting formulation in various applications.

We focus on literature that combines tax and subsidy setting, i.e.,\enquote{tax revenue recycling}. \cite{qiu2020carbon2} evaluate a carbon tax and rebate system for air passenger transportation, where airlines pay taxes on emissions and receive subsidies for mitigation improvements. This approach reduces carbon emissions under low transaction costs and fuel price differences. \cite{jiang2021aviation} analytically evaluates the effects of using aviation tax revenues to subsidize high-speed rail. 
Counter-intuitively, he found that the policy can lead to undesired effects where air traffic volume increases and high-speed rail volume decreases. 
Because limiting the analysis to a single origin-destination pair abstracts the details of delivery such as multiple pickups and delivery points, consolidation effect, time window, \cite{qiu2020carbon} explore a more complex pollution routing problem with up to 100 nodes. 

They propose using road tax revenues for funding allowances based on reduced road freight emissions. This policy effectively reduces emissions while controlling additional costs for freight forwarders.
Our work extends upon this stream of literature by investigating the solution structure of our bi-level formulation, allowing us to evaluate optimal policy settings and by further numerical analyses. Unlike \cite{qiu2020carbon2, jiang2021aviation}, we consider a more realistically sized problem instead of a single OD pair. Additionally, we investigate the solution structure of our bi-level formulation, allowing us to evaluate optimal policy settings.

Very few studies consider the interaction between stakeholders in the context of using the scheduled line services for inner-city freight transport. \cite{ma2022game} propose an analytical model to characterize the strategic interaction between metro operators and logistics companies in a metro-integrated logistics system. The model considers both cooperative and non-cooperative markets. The metro operator sets the freight price, while logistics companies choose their transportation plan, opting for either road transport or metro modes. They show that the metro-integrated system can benefit the operators and logistics companies. In an extension, \cite{ma2023urban} additionally permit outsourcing where a freight forwarder hires a freight carrier for the road freight transport. However, both studies analyzed a stylized model using a single origin-destination pair. Compared to their work, we integrate the operational decision-making of the freight forwarder as a response to the policy. 

\subsection{Last-mile delivery using Scheduled Lines}
Several studies have investigated planning to use scheduled line services for inner-city freight transportation. The scheduled line services may include inland waterways, bus, tram, or metro services. Since these services are commonly passenger-oriented, the literature also refers to these services as freight-on-transit. For a comprehensive overview of the literature on freight-on-transit, we refer the reader to recent literature reviews by \cite{cleophas2019collaborative},   \cite{elbert2022freight} and \cite{cheng2023integrated}. We focus our review on the operational level studies on the last-mile delivery using the scheduled line services. For synchronization of routing with schedules in other applications, please refer to \cite{SOARES2024817}. Two primary approaches model the operation planning of this application: the two-echelon system and the Pickup and Delivery Problem with Time Windows and Scheduled Lines (PDPTW-SL).

The two-echelon system involves delivering freight from distribution centers to city areas using scheduled line services, followed by last-mile delivery via city freighters from public transport stations. Key studies in this area include \cite{masson2017optimization,schmidt2022using}, and \cite{mo2023vehicle}, which abstract the first echelon decisions and model them as replenishment nodes with fixed demands. They address the problem as a pickup and delivery problem with time windows, employing various algorithms like Adaptive Large Neighborhood Search (ALNS) and Branch-and-Price-and-Cut. \cite{wang2024two} take a different approach by explicitly modeling the first echelon routing decision, connecting satellites via public transport, and using city freighters for the second echelon, i.e., for the last-mile delivery.

The PDPTW-SL approach, proposed by \cite{ghilas2016pickup}, involves freight forwarders organizing last-mile delivery by choosing between vehicles alone or a mix of vehicles and public transit systems. Their computational results on small instances show that public transport can reduce operational costs by up to 20\%. To handle larger instances, \cite{ghilas2016adaptive} developed an Adaptive Large Neighborhood Search (ALNS) algorithm capable of solving instances with up to 100 transportation requests. Additionally, \cite{ghilas2018branch} proposed a Branch-and-Price method to solve instances with up to 50 transportation requests. \cite{de2024sustainable} relaxed the capacity constraints of public transport and introduced a destroy-and-repair neighborhood search heuristic to handle up to 500 requests. Similarly, \cite{he2023optimization} extended the PDPTW-SL to allow multiple trips for vehicles and incorporated practical constraints such as driver workload and driving distance limits. Others adopt the sample average approximation to tackle the stochastic version of the PDPTW-SL with stochastic freight demands \citep{ghilas2016scenario} and scheduled line capacity \citep{mourad2021integrating}.

However, these studies assume tactical decisions of scheduled line services, such as scheduling or capacity, and focus solely on the freight forwarder’s perspective, neglecting interactions with other stakeholders like transportation authorities. The literature contains few tactical studies on inner-city freight transport, including train scheduling \citep{ozturk2018optimization,horsting2023scheduling} and station capacity determination \citep{fontaine2021scheduled}. This gap hinders a comprehensive understanding of stakeholder interactions. In this study, we address this by incorporating the role of transportation authorities in setting road taxes and subsidies for scheduled line services and investigating the freight forwarder’s response in the PDPTW-SL setting.

\section{Problem Description and Formulation}\label{sec:ProblemStatement}
This section first presents the problem narrative of setting road tax and scheduled line subsidy and then the mathematical formulation. Next, we present some properties of the proposed problem and its optimal solutions.

\subsection{Problem Narrative}
We consider the problem of determining a tax and subsidy policy to incentivize a modal shift away from emission-intense road transport. The transportation authority sets fuel or access taxes per unit driven distance for road vehicles. It subsidizes freight transportation on scheduled lines such as urban light rail or inland waterway transportation. Given tax and subsidy levels, a freight forwarder directly transports goods with a truck or utilizes scheduled lines to minimize total costs. Due to the leader-follower structure, we model the policy-setting problem as a Stackelberg game.

The transportation authority aims to promote modal shift. In most European cities, public authorities are in charge of public transport operations and can levy local taxes. After passenger transport, scheduled lines have a remaining capacity that can be used for freight transport at a predefined cost. 
The authority has a budget of $B$ available to subsidize the scheduled line, reducing the fare by $s$. Further subsidies must be financed through distance-based road access taxes at the rate of $t$. 

The freight forwarder builds routes to transport exogenous requests, including a pickup location, delivery location, and time window. Freight demands remain constant regardless of tax and subsidy levels since the freight forwarder maintains consistent customer pricing. This is typically the case in competitive markets, where freight forwarders are, in essence, price-takers. The freight forwarder will utilize the scheduled line for an order if it reduces transportation costs and is feasible concerning the time windows. In this case, the good is transported to a transshipment point by a truck, waits there for the next available scheduled line (subject to its capacity), transported to the next transshipment point, and picked up by a truck. 
The routes depend on the authority's policy $\left(s,t\right)$. 

Since this paper focuses on the operational level, we consider strategic and tactical decisions regarding the scheduled line service, such as network design and operating schedule, as given.  We acknowledge that the effectiveness of the delivery scheme also depends on these decisions, as demonstrated by \cite{ghilas2016pickup}. 


In the following, we first detail a freight forwarder's decision problem and then formulate the bi-level tax-and-subsidy setting problem. 

\subsection{Freight Forwarder Decision Model} \label{subsec:FreightForwarderDeciaionModel}

The freight forwarder decides how to ship their transportation requests through their network, using direct transportation by truck or transshipping some requests via a scheduled line. 
The freight forwarder's joint decision for all their requests can be modeled as a  Pickup and Delivery Problem with Time Windows and Scheduled Line (PDPTW-SL), as introduced by \cite{ghilas2016pickup}. We only provide a conceptual description of the model as our bi-level formulation applies to other freight forwarder models involving dual transportation channels.
%
The decisions include the routing plan for road vehicles and the shipment amount allocated to the scheduled line. 
Constraints include typical routing and flow constraints from the PDPTW problem, accounting for vehicle capacity limits and time windows for requests. Additionally, we consider the scheduled line capacity and service’s running schedule, ensuring time synchronization between road vehicles and scheduled line services.
The objective is to minimize transportation costs composed of road and scheduled line costs. 
Unlike the approach in \cite{ghilas2016pickup}, the freight forwarder considers tax as part of the routing costs, and the scheduled line is partially subsidized. We will explain this further in the following section.

\subsection{Bilevel Road Tax and Train Subsidy Model}\label{subsec:TaxSubsidyModel}

The upper-level variables are $s$ for the scheduled line subsidy and $t$ for the road tax. The lower-level routing variables are succinctly notated as $\mathbf{x}$. For brevity, we introduce two other lower-level decision variables that directly depend on $\mathbf{x}$: the total road distance $d$ and the total flow cost on the scheduled line $f$. $\mathcal{X}$ gives the set of feasible lower-level solutions. We assume that even if multiple feasible lower-level solutions exist for one input $s,t$, the freight forwarder will select the same solution deterministically. The freight forwarder's strategy can either be cooperative \citep[see, e.g.~][]{dempe2024bilevel} or adversarial \citep[see, e.g.~][]{tsoukalas2009global,liu2018pessimistic}. The optimal values of the lower-level variables are indicated with a $^*$ and as a function of $s$ and $t$ since they depend on the upper-level variables. The rest of the parameters are the unit distance cost, $\phi$, and the available budget of the authority, $B$.
Then, the bi-level model is given by:
\begin{align}
     \min \limits_{s,t} \quad   & d^*(s,t) \label{eq:objective}\\
\textrm{subject to} \quad s f^*(s,t) - t \phi d^*(s,t)  & = B &\label{eq:constr_budget} \\
0 \leq s & \leq 1, \quad 0\leq t, &\label{eq:constr_domain}  \\
%
\left(d^*(s,t), f^*(s,t),x^*(s,t) \right) & \in  \argmin \limits_{(d,f,x)\in \mathcal{X}}  \left(1 + t \right) \phi d + \left(1 - s \right) f.  &\label{eq:constr_feasbile_area}  
\end{align}

Equation \eqref{eq:objective} represents the objective of the transportation authority: minimizing the total road distances. Constraint \eqref{eq:constr_budget} defines the budget constraint, where the difference between total subsidy and total tax revenue equals the given budget. Constraint \eqref{eq:constr_domain} specifies the subsidy and tax ranges. Finally, constraint \eqref{eq:constr_feasbile_area} defines that $\left(d^*(s,t), f^*(s,t),x^*(s,t) \right)$ are the optimal solutions for the lower-level problem introduced in subsection \ref{subsec:FreightForwarderDeciaionModel} with the objective function under tax and subsidy policy. Moreover, with the term $(1-s)$, we ensure that the provided subsidy does not exceed the total flow cost on the scheduled line service to prevent unnecessary modal shifts.

\section{Properties of the optimal policy}
\label{sec:PropertiesOptimal}
We begin by outlining the properties of the model defined in section \ref{subsec:TaxSubsidyModel}. Subsequently, we derive a condition for the optimal policy and explore its associated properties. All proofs are provided in \ref{app:proofs}. For notational brevity, we assume that $\phi =1$. 

We first examine the effects of changing $s$ and $t$, subject to a budget of $B$. Proposition \ref{prop:Proposition_1} shows that the objectives of the transportation authority and freight forwarders' objectives are inherently conflicting; it is impossible to improve one party, worsening the other.  Then, we show that increasing taxes cannot increase the driven distance (Proposition \ref{prop:Proposition_2}). On the other hand, Example \ref{ex:subsidyUpDistanceUp} shows that increasing the subsidy level can increase distances. Notwithstanding, we show that if the transportation authority sets the tax rate optimally, driven distance does decrease in the subsidy level (Proposition~\ref{prop:newProp3}), such that an optimal solution exists that fully subsidizes the scheduled line (Proposition \ref{prop:newProp4}). Moreover,  Proposition \ref{prop:Proposition_4} specifies the impact of increasing the authority's budget on the cost of the freight forwarder.

\begin{proposition}\label{prop:Proposition_1}

For a given budget $B$, the freight forwarder's total cost decreases if and only if the driven distance is increasing, i.e., for two feasible solutions $\langle s_1,t_1,d_1^\star,f_1^\star\rangle$, $\langle s_2, t_2,d_2^\star,f_2^\star\rangle$, 
\begin{align}\label{eq:costFreightForwarder}
d_1^\star<d_2^* \iff \left(1 + t_1\right) d_1^\star + \left(1 - s_1\right) f_1^\star &> \left(1 + t_2\right) d_2^\star + \left(1 - s_2\right) f_2^\star 
\end{align}
\end{proposition}\
The following proposition shows that, with a fixed budget, the transportation authority cannot worsen its objective by increasing the tax rate. 

\begin{proposition}
\label{prop:Proposition_2}
For a given budget $B$, the driven distance is non-increasing in the tax level $t$, i.e., for two feasible solutions $\langle s_1,t_1,d_1^\star,f_1^\star\rangle$, $\langle s_2,t_2,d_2^\star,f_2^\star\rangle$ with $t_1 < t_2$, it holds that 
$d_1^\star \geq d_2^\star$. 
\end{proposition}

Note that it is impossible to increase the tax level indefinitely to keep reducing driven distance since the solution should remain feasible, and the subsidy level can be at most 1. We have shown that increasing the tax level reduces the driven distance but increases the freight forwarder's costs. 
On the other hand, increasing the subsidy level, $s_1 < s_2$, at a given budget $B$ \textit{can} result in increased distances, $d_1^\star < d_2^\star$, thereby counteracting the goals of the policymaker, as Example \ref{ex:subsidyUpDistanceUp} shows. 
\begin{example}\label{ex:subsidyUpDistanceUp}
Assume that the budget $B=0$ and that the lower-level problem has only two feasible routing solutions ($\mathcal{X} = \{\langle d_1 = 15, f_1 = 20\rangle, \langle d_2 = 20, f_2 = 5\rangle\}$) due to the time window, vehicle capacity, and other constraints. Under the policy $(s_1=0.5,t_1=2/3)$, the first solution is optimal for the freight forwarder, resulting in a distance 15. However, under the policy $(s_2=0.6,t_2=0.15)$, with a higher subsidy level \textit{but lower tax}, the second solution is optimal for the freight forwarder, resulting in distance 20. 
\end{example}

However, we can show that if the transportation authority selects the tax rate optimally given the subsidy level, the driven distance \textit{does} decrease in the subsidy, as shown by Proposition~\ref{prop:newProp3}. It immediately follows that the transportation authority can minimize distance by fully subsidizing the scheduled line, as formalized in Proposition~\ref{prop:newProp4}.

\begin{proposition}
\label{prop:newProp3}
For a given budget $B$ and subsidy level $s$, let $d^*(s)$ denote the lowest driven distance over all tax rates $t$. Then, $d^*(s)$ decreases monotonically in $s$. 
\end{proposition}

\begin{proposition}
    
\label{prop:newProp4}
Let $f^\text{full}$ denote the scheduled line costs of the freight forwarder under the policy $(s=1,t=0)$. If $B\leq f^\text{full}$, there exists an optimal where $s=1$. If $B>f^\text{full}$, the problem is infeasible. 
\end{proposition}

After identifying the structure of an optimal policy, we now turn to the efficiency of the budget. We show that the freight forwarder's routing decision is independent of the budget if $s=1$ and that an increase in budget directly and by an equal amount decreases the freight forwarder's cost (Proposition \ref{prop:Proposition_4}). 

\begin{proposition}    
\label{prop:Proposition_4}
Under the optimal policy, increasing the budget does not influence the freight forwarder's optimal routing decision but decreases its total cost, i.e., for any two optimal policies with $s_1 = s_2 = 1$ with their corresponding budget $B_1 < B_2 \leq f^\text{full}$ and total freight forwarder cost $C_1, C_2$, it holds that
\begin{align*}
    B_1 + C_1 = B_2 + C_2
\end{align*}
\end{proposition}

\section{Solution algorithm}\label{sec:Algorithm}

Given a policy $(s,t)$, we solve the lower-level problem using the ALNS algorithm for PDPTW-SL developed in \citep{ghilas2016adaptive}. We are aware that a Branch-and-Price algorithm for PDPTW-SL is proposed by  \cite{ghilas2018branch}, but use ALNS to be able to solve larger problem sizes. Since the orders transported by the freight forwarder may vary daily, we aggregate the results of various demand scenarios to compute the total subsidy and tax revenue. 

The optimal policy for the transportation authority is computed according to Proposition~\ref{prop:newProp4}. We also perform tests that do not use a full subsidy. Given some subsidy level $s<1$, we compute the tax rate that satisfies the budget constraint using a bisection search. We start with initial guesses for the tax rate that, respectively, undershoot and overshoot the budget constraint and continue refining these guesses until the constraint is sufficiently met. For an algorithm overview, refer to Algorithm \ref{alg:bisection}. 

In Section~\ref{sec:PropertiesOptimal}, we showed that multiple feasible tax rates may exist for a given subsidy level, implying that Algorithm~\ref{alg:bisection} may result in a suboptimal solution. However, the algorithm does perform well in numerical experiments, as we will see in the next section.  
\begin{algorithm}
\caption{Tax rate search}\label{alg:bisection}
\begin{algorithmic}[1]
    \State \textbf{Initialize} variables $a$, $b$, and $\epsilon$ (where $a$ and $b$ define the tax interval, and $\epsilon$ is the tolerable error).
    \State \textbf{Input} initial guesses $x_0$ and $x_1$ such that $f(x_0) \cdot f(x_1) < 0$. ($f(\cdot)$ - solving scenarios of lower problems using ALNS given a tax value and return average budget)
    \State \textbf{Initialize} maximum number of iterations $N$.
    \State $n \gets 0$ \Comment{Initialize iteration counter}
    \While{$\left| f\left(\frac{{x_0 + x_1}}{2}\right)\right| > \epsilon$ and $n < N$ and $f\left(\frac{{x_0 + x_1}}{2}\right) \neq B $} \Comment{B - Total budget}
        \State $x_2 \gets \frac{{x_0 + x_1}}{2}$
        \If{$f(x_0) \cdot f(x_2) < 0$}
            \State $x_1 \gets x_2$
        \Else
            \State $x_0 \gets x_2$
        \EndIf
        \State $n \gets n + 1$ \Comment{Increment iteration counter}
    \EndWhile
\end{algorithmic}
\end{algorithm}
\section{Numerical Experiments}\label{sec:NumericalExp}
This section presents the findings from a series of numerical experiments with the bi-level model. We describe the experimental design and then discuss the benefits and consequences of optimal policies for transportation authorities and freight forwarders. Moreover, we conduct a sensitivity analysis to assess the impact of the scatteredness of order and scheduled line frequency on optimal policy settings, transportation authority, and freight forwarder objectives. Moreover, we explore the trade-off between total driven distance and the cost of freight forwarders for policies with different subsidy levels. Finally, we conclude the section with a case study based on the city of Berlin.
\subsection{Experimental Design}
We derive our instances from the largest instances used by \cite{ghilas2016adaptive}, with 100 transportation and three scheduled line stations. Compared to the original instances, we assume all stations also serve as vehicle depots. We generate instances that vary across three dimensions: order location geography, order allocation, and time window length. Orders are always sampled within a certain radius of the three stations. Still, with the\textit{Intercity} (Inter) geography, the scheduled line stations are twice as far apart as with the \textit{Metropolitan} (Metro) cases, as illustrated in Figure~\ref{fig:gen_instance}.

\begin{figure}
    \centering
    \begin{subfigure}[b]{0.48\linewidth}        
        \centering
        \includegraphics[width=\linewidth]{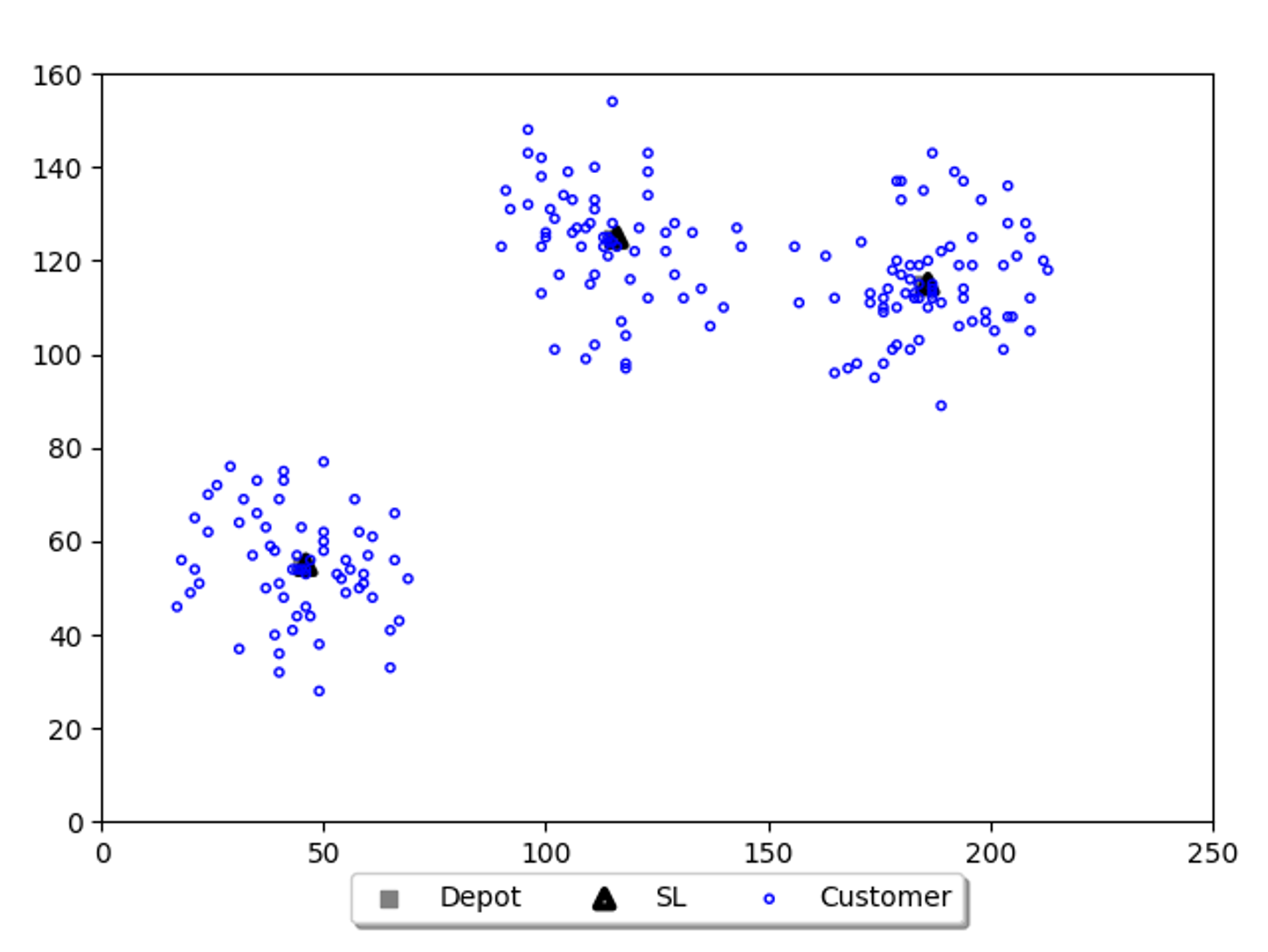}
        \caption{Intercity}
        \label{fig:intercity}
    \end{subfigure}
    \begin{subfigure}[b]{0.48\linewidth}        
        \centering
        \includegraphics[width=\linewidth]{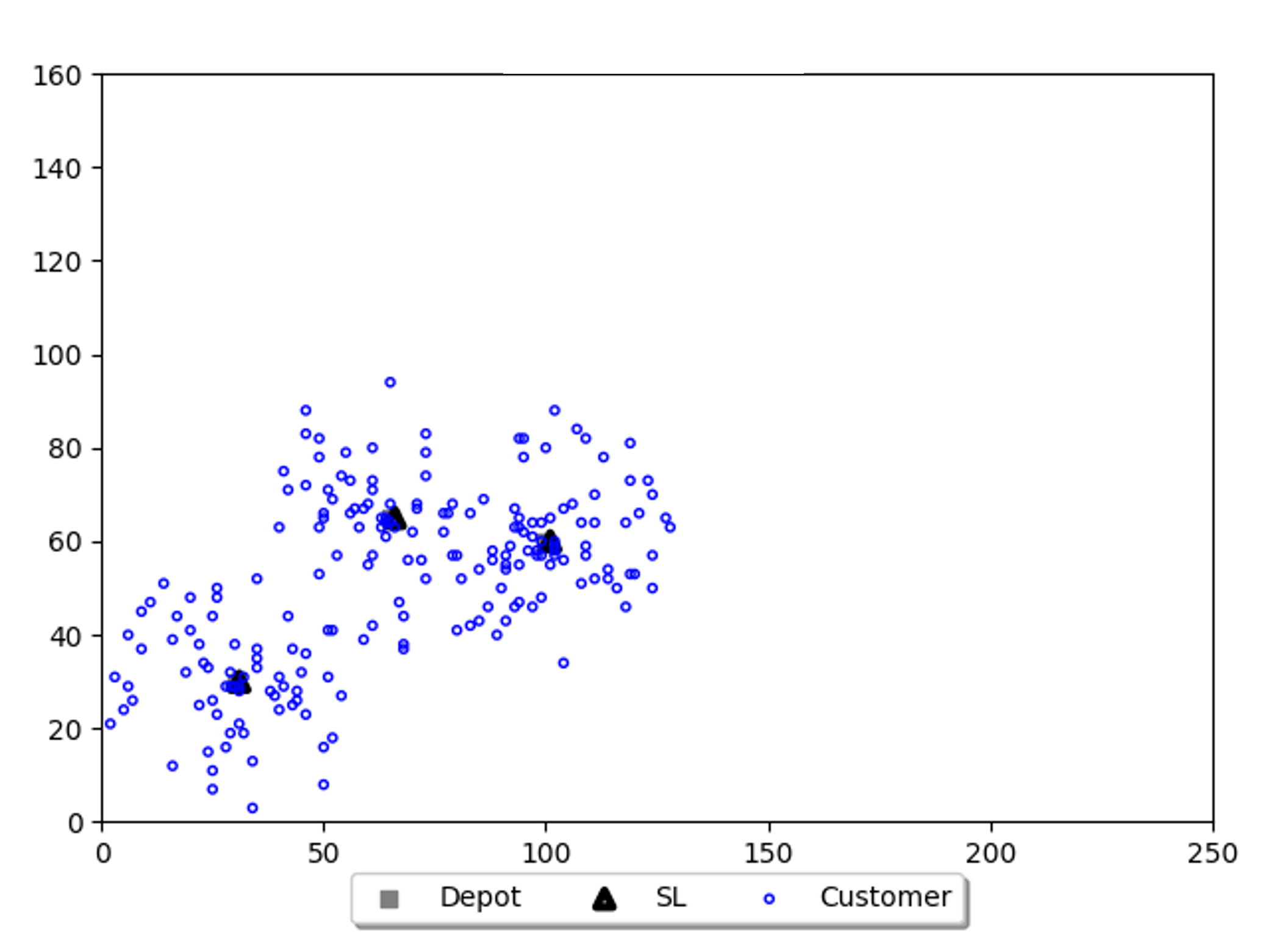}
        \caption{Metropolitan}
        \label{fig:metro}
    \end{subfigure}
    \caption{Different order location geographies}
    \label{fig:gen_instance}
\end{figure}


‘Order allocation’ defines the process of matching pickup and delivery locations. We examine two cases: ‘Different’ (Diff) and ‘Random’ (Rand). In the' Different' scenario, we pair pickup and delivery locations from whose closest scheduled line station is different. Conversely, ‘Random’ involves a selection process for pickup and delivery locations. We also distinguish the time window length into two cases: ‘Tight’ (T) and ‘Wide’ (W), representing 45 and 60 time units, respectively.

To ensure fair cost comparison, the scheduled line cost per demand varies based on the distance between transferred stations and the vehicle cost to cover the same distance, preventing the scheduled line service from being significantly cheaper than the vehicle cost. The cost per unit distance is 0.25, and the cost per unit of freight on the train is 0.1 per unit distance.

\subsection{Optimal Policies for the Transportation Authority and the Freight Forwarder}
We identify the benefits and consequences of the optimal policy by comparing the results under the optimal policy to the base scenario where there is no intervention.  Table \ref{tab:optOpCondition} compares the base and optimal policy on modal shift level and objective of transportation authority and freight forwarder. Since the authority's budget is zero, the tax columns show the required tax level to achieve the total scheduled line subsidy. Other columns show the average value of performance metrics obtained from ten randomly generated scenarios. This approach ensures the metrics are not overestimated and account for any variability due to scenario-specific errors.


\begin{table}[h]
\centering

\caption{Base versus optimal (Opt) policy under budget-neutral conditions ($B=0$)}
\label{tab:optOpCondition}
\scalebox{0.8}{
\begin{tabular}{lrrrrrrrrrrrr}
\hline
 & \multicolumn{3}{c}{Driving distance} &  & \multicolumn{2}{c}{\%modal shift} &  & \multicolumn{3}{c}{Operation cost} &  &  \\ \cline{2-4} \cline{6-7} \cline{9-11}
Instance & \multicolumn{1}{c}{Base} & \multicolumn{1}{c}{Opt} & \multicolumn{1}{c}{\%} &  & \multicolumn{1}{c}{Base} & \multicolumn{1}{c}{Opt} &  & \multicolumn{1}{c}{Base} & \multicolumn{1}{c}{Opt} & \multicolumn{1}{c}{\%} &  & \multicolumn{1}{c}{tax} \\ \cline{1-4} \cline{6-7} \cline{9-11} \cline{13-13} 
Inter-Diff-T & 7094.0 & 6414.4 & -9.6 &  & 8.9 & 40.1 &  & 1968.0 & 2451.8 & 24.6 &  & 0.53 \\
Inter-Diff-W & 6290.0 & 5409.0 & -14.0 &  & 9.4 & 59.7 &  & 1772.2 & 2573.8 & 45.2 &  & 0.90 \\
Inter-Rand-T & 6406.3 & 5785.5 & -9.7 &  & 6.2 & 28.4 &  & 1728.5 & 2035.3 & 17.8 &  & 0.41 \\
Inter-Rand-W & 5592.4 & 4755.2 & -15.0 &  & 10.0 & 46.1 &  & 1601.3 & 2161.9 & 35.0 &  & 0.82 \\
Metro-Diff-T & 4481.4 & 4253.6 & -5.1 &  & 8.7 & 39.2 &  & 1210.3 & 1473.8 & 21.8 &  & 0.39 \\
Metro-Diff-W & 4105.2 & 3840.5 & -6.4 &  & 6.1 & 49.2 &  & 1091.0 & 1483.4 & 36.0 &  & 0.55 \\
Metro-Rand-T & 4170.4 & 3990.1 & -4.3 &  & 6.4 & 29.7 &  & 1109.3 & 1304.3 & 17.6 &  & 0.31 \\
Metro-Rand-W & 3809.9 & 3570.8 & -6.3 &  & 5.7 & 43.1 &  & 1009.7 & 1344.8 & 33.2 &  & 0.51 \\ \hline
\end{tabular}}
\end{table}
\begin{result}
The optimal policy results in high road distance savings due to substantial freight modal shifts, particularly in the wide time window and distinct order pair instances.
\end{result}

Table \ref{tab:optOpCondition} shows that the optimal policy reduces driving distances by 4.3 \% to 12.5 \%  across all instances. This reduction is attributed to the increased modal shift from road vehicles to the scheduled line. Additionally, freight forwarders can achieve greater shifts when the time window is wide, and the pickup and delivery nodes are in different clusters. A wide time window provides the extra travel time needed for the modal shift. Furthermore, when order pairs are in different clusters, the distances are greater, leading to higher cost savings from the modal shift.

\begin{table}[h]
\centering
\caption{Max vehicle capacity and \# Vehicles under Base versus optimal (Opt) policy under budget-neutral conditions ($B=0$)}
\label{tab:optOpCondition2}
\scalebox{0.8}{
\begin{tabular}{lrrrrrr}
\hline
 & \multicolumn{3}{c}{\begin{tabular}[c]{@{}c@{}}Max capacity \\ utilization\end{tabular}} & \multicolumn{1}{l}{} & \multicolumn{2}{c}{\#Vehicles} \\ \cline{2-4} \cline{6-7} 
Instance & \multicolumn{1}{c}{Base} & \multicolumn{1}{c}{Opt} & \multicolumn{1}{c}{\%} & \multicolumn{1}{l}{} & \multicolumn{1}{c}{Base} & \multicolumn{1}{c}{Opt} \\ \cline{1-4} \cline{6-7} 
Inter-Diff-T & 9.5 & 9.4 & -1.1 &  & 21.9 & 26.8 \\
Inter-Diff-W & 11.0 & 9.4 & -14.5 &  & 19.6 & 27.8 \\
Inter-Rand-T & 8.3 & 7.9 & -4.8 &  & 20.3 & 24.2 \\
Inter-Rand-W & 10.3 & 9.0 & -12.6 &  & 18.5 & 24.4 \\
Metro-Diff-T & 9.5 & 9.0 & -5.3 &  & 17.3 & 24.7 \\
Metro-Diff-W & 11.4 & 10.1 & -11.4 &  & 15.0 & 24.7 \\
Metro-Rand-T & 9.0 & 8.8 & -2.2 &  & 15.9 & 21.3 \\
Metro-Rand-W & 10.6 & 9.4 & -11.3 &  & 14.9 & 22.9 \\ \hline
\end{tabular}}
\end{table}

\begin{result}
Although the optimal policy requires more vehicles, freight forwarders can utilize smaller ones.
\end{result}

Table 2 shows the maximum load and the number of vehicles for each instance under both the base and optimal policies. Under the optimal policy, freight forwarders use more vehicles to accommodate the increased modal shift. This increase may be due to time windows, capacity, and spatial patterns of requests. This finding aligns with \cite{ghilas2016adaptive}, which observed that instances with higher modal shifts, such as clustered instances, required more vehicles under the PDPTW-SL scheme. However, the maximum vehicle capacity utilization decreases, especially when the time window is wide and the instances involve different clusters. Consequently, freight forwarders can use smaller vehicles.

\begin{result}
Under the optimal policy, freight forwarders face higher operational costs due to increased tax levels, particularly in instances with wide time windows.
\end{result}

As shown in Table \ref{tab:optOpCondition}, the percentage increase in operational costs is higher with elevated tax levels, ranging from 17.6 \% to 46.1 \%. This increase is quite high in wide time window settings, where the modal shift is high. However, the transportation authority can offset this increase using its budget (Proposition \ref{prop:Proposition_4}).

\subsection{Order Scatteredness and Scheduled Line Services Frequency}

The effectiveness of the PDPTW-SL scheme is influenced by the spatial distribution of requests and the configuration of scheduled line services \citep{ghilas2016pickup}. Consequently, it is important to assess the impact of these factors on optimal policy implementation. This section evaluates the effects of order scatteredness and scheduled line frequency on policy settings and stakeholders' objectives. From Table \ref{tab:optOpCondition}, we observe the highest difference in the modal shift between the intercity and metropolitan with the different clusters and wide time windows. Therefore, we choose this instance for further analysis.

Order scatteredness refers to distances between the pickup and delivery points. We vary the scatteredness level ($k$) by multiplying the original coordinates by $(k/2+0.5)$. The Intercity and Metropolitan cases correspond to $k$ equal 1 and 0, respectively.  For each scatteredness level, we generate 10 sets of 10 scenarios. 

Figure \ref{fig:SensitivityGeo} illustrates the impact of order scatteredness by comparing the performance of base and optimal policies.
Each point represents the average value of the performance metrics from the 10 sets of instances. We found that instances become infeasible if $k>$1.2 due to the time windows. As order scatteredness increases, the optimal modal shift decreases from 98.3 \% to 38.3 \% due to the time window limit. As orders are farther apart, some orders are not feasible for a detour via transferred nodes. On the other hand, under the base policy, the modal shift gradually increases, becomes levels off, and then decreases in the end. In the first phase, it is likely that as the orders scatter, the modal shift provides cost savings until the time window becomes a limitation. Despite this decrease in the modal shift, the distance saving increases until the scatteredness level of 0.8, after which the saving decreases. For the policy settings, the tax trends resemble the optimal modal shift trend. In addition, the system's total cost increases with the scatteredness level. Under the optimal policy, the rate of increased cost seems to be slower with the decreasing modal shift.

\begin{figure}[h!]
    \centering
    
        \begin{subfigure}[b]{0.49\textwidth}
        \centering
        \begin{adjustbox}{max width=\textwidth}
        \begin{tikzpicture}
        \begin{axis}[
            xlabel={Scatteredness},
            ylabel={\% Modal Shift},
            xmin=0, xmax=1.2,
            ymin=0, ymax=100,
            xtick={0,0.2,0.4,0.6,0.8, 1,1.2},
            ytick={0,20,...,100},
            legend pos=north east,
            ymajorgrids=true,
            grid style=dashed,
        ]
        \addplot[
            color=blue,
            mark=*,
            sharp plot,
            ]
            coordinates {
(0,1.19975131665742)
(0.2,2.06653510032055)
(0.4,4.1926405680628)
(0.6,6.23247045161245)
(0.8,11.3345165701958)
(1,11.2298528757981)
(1.2,7.75448250601095)
            };
        \addplot[
            color=orange,
            mark=*,
            sharp plot,
            ]
            coordinates {
(0,96.4926593997625)
(0.2,96.7776741274459)
(0.4,94.225665237957)
(0.6,87.9227891185136)
(0.8,75.3765103254051)
(1,58.9809723018312)
(1.2,38.3408391549258)
            };
        \legend{Base, Opt}
        \end{axis}
        \end{tikzpicture}
        \end{adjustbox}
        \caption{Scatteredness - \% Modal Shift}
    \end{subfigure}
     \hfill
    \begin{subfigure}[b]{0.49\textwidth}
        \centering
        \begin{adjustbox}{max width=\textwidth}
        \begin{tikzpicture}
        \begin{axis}[
            xlabel={Scatteredness},
            ylabel={Total Driving Distance (x1000)},
            xmin=0, xmax=1.2,
            ymin=1.5, ymax=8.1,
            xtick={0,0.2,0.4,0.6,0.8, 1, 1.2},
            ytick={1.5, 2.5,...,7.5 },
            legend pos=north west,
            ymajorgrids=true,
            grid style=dashed,
        ]
        \addplot[
            sharp plot,
            color=blue,
            mark=*,
        ]
        coordinates {
(0,2.10425516056215)
(0.2,2.69429837092931)
(0.4,3.40148965244647)
(0.6,4.19013861886712)
(0.8,5.08288845087076)
(1,6.24238818124906)
(1.2,8.02163348186987)
        };
        \addplot[
            sharp plot,
            color=orange,
            mark=*,
        ]
        coordinates {
(0,1.6962901131671)
(0.2,2.11126304002178)
(0.4,2.6593710641595)
(0.6,3.25635396767146)
(0.8,4.18223388088091)
(1,5.35612132161892)
(1.2,7.18741036717496)
        };
        \legend{Base, Opt}
        \end{axis}
        \end{tikzpicture}
        \end{adjustbox}
        \caption{Scatteredness - Total Driving Distance}
    \end{subfigure}
    
    \vspace{1cm}
    
       \begin{subfigure}[b]{0.49\textwidth}
        \centering
        \begin{adjustbox}{max width=\textwidth}
        \begin{tikzpicture}
        \begin{axis}[
            xlabel={Scatteredness},
            ylabel={Tax},
            xmin=0, xmax=1.2,
            ymin=0.4, ymax=2.5,
            xtick={0,0.2,...,1.2},
            ytick={0.4, 0.8, ..., 2.6},
            legend pos=north west,
            ymajorgrids=true,
            grid style=dashed,
        ]
        \addplot[
            color=blue,
            mark=*,
            sharp plot,
            ]
            coordinates {
(0,2.45060767189936)
(0.2,2.37290746417275)
(0.4,2.1461869785992)
(0.6,1.88058323682267)
(0.8,1.40514673270738)
(1,0.941257985333615)
(1.2,0.48796847655174)
            };
        \end{axis}
        \end{tikzpicture}
        \end{adjustbox}
        \caption{Scatteredness - Tax}
    \end{subfigure}
    \hfill
    \begin{subfigure}[b]{0.49\textwidth}
        \centering
        \begin{adjustbox}{max width=\textwidth}
            
        \begin{tikzpicture}
        \begin{axis}[
            xlabel={Scatteredness},
            ylabel={Total Cost of the System(x1000)},
            xmin=0, xmax=1.2,
            ymin=0.5, ymax=2.8,
            xtick={0,0.2,0.4,0.6,0.8, 1,1.2},
            ytick={1,1.5,..., 2.5},
            legend pos=north west,
            ymajorgrids=true,
            grid style=dashed,
        ]
        \addplot[
            sharp plot,
            color=blue,
            mark=*,
        ]
        coordinates {
(0,0.537321233127331)
(0.2,0.697274020873315)
(0.4,0.909029705420678)
(0.6,1.15003158205593)
(0.8,1.48608904355631)
(1,1.80909819162004)
(1.2,2.19193750996833)
        };
        \addplot[
            sharp plot,
            color=orange,
            mark=*,
        ]
        coordinates {
(0,1.46330652395653)
(0.2,1.78004320189908)
(0.4,2.09165336690269)
(0.6,2.34497817530704)
(0.8,2.51380486786007)
(1,2.59898734880068)
(1.2,2.67327978076314)
        };
        \legend{Base, Opt}
        \end{axis}
        \end{tikzpicture}
        \end{adjustbox}
        \caption{Scatteredness - Total cost of the system}
    \end{subfigure}

    \caption{Impact of order scatteredness on the performances of base and optimal policies}
    \label{fig:SensitivityGeo}
\end{figure}
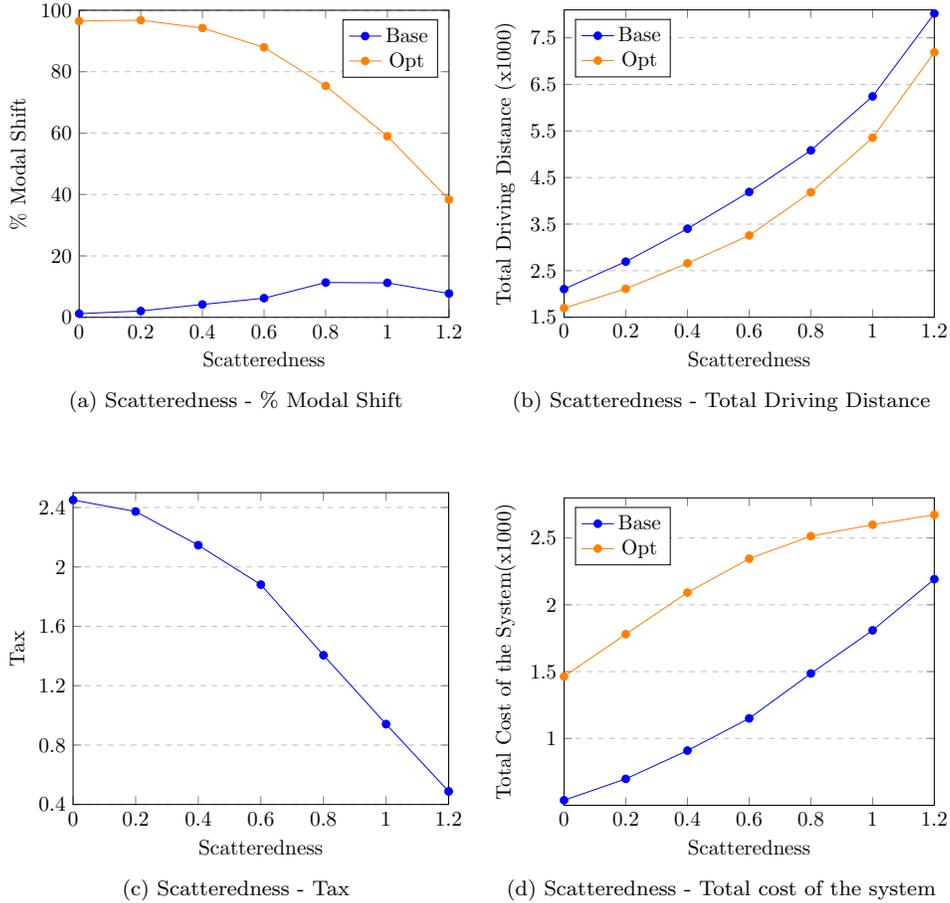

Next, we vary the frequency of scheduled line services to assess their impact on policy settings and stakeholder objectives. Similar to the previous sensitivity analysis, we use the same instance, varying scheduled line frequency from 1 to 10 services per hour. For each frequency level, we generate 10 sets of instances with ten scenarios each. Figure \ref{fig:SensitivityFreq} shows the results for all performance measures, with each point representing the average value of the instances.

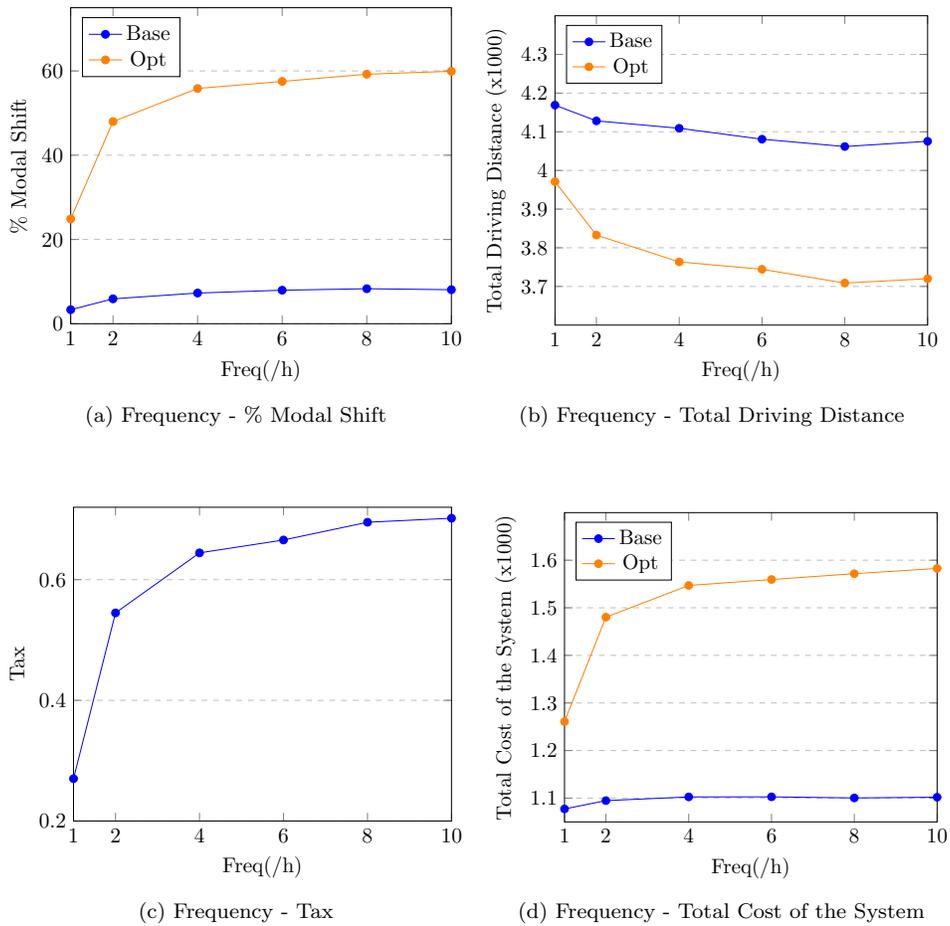
\begin{figure}[h!]
    \centering
    \begin{subfigure}[b]{0.49\textwidth}
        \centering
        \begin{adjustbox}{max width=\textwidth}
        \begin{tikzpicture}
        \begin{axis}[
            xlabel={Freq(/h)},
            ylabel={\% Modal Shift},
            xmin=1, xmax=10,
            ymin=0, ymax=75,
            xtick={1,2, 4, 6, 8, 10},
            ytick={0,20,40,60},
            legend pos=north west,
            ymajorgrids=true,
            grid style=dashed,
        ]
        \addplot[
            color=blue,
            mark=*,
            sharp plot,
            ]
            coordinates {
(1,3.32317495828447)
(2,5.91142471801031)
(4,7.28592064432513)
(6,7.9550490980951)
(8,8.31154111476017)
(10,8.07573437455817)
            };
        \addplot[
            color=orange,
            mark=*,
            sharp plot,
            ]
            coordinates {
(1,24.863901817233)
(2,47.9750007931577)
(4,55.8364461628367)
(6,57.5150762945711)
(8,59.2239044745461)
(10,59.9132890256991)
            };
        \legend{Base, Opt}
        \end{axis}
        \end{tikzpicture}
        \end{adjustbox}
        \caption{Frequency - \% Modal Shift}
    \end{subfigure}
        \hfill
    \begin{subfigure}[b]{0.49\textwidth}
        \centering
        \begin{adjustbox}{max width=\textwidth}
        \begin{tikzpicture}
        \begin{axis}[
            xlabel={Freq(/h)},
            ylabel={Total Driving Distance (x1000)},
            xmin=1, xmax=10,
            ymin=3.6, ymax=4.4,
            xtick={1, 2, 4, 6, 8, 10},
            ytick={3.7,3.8 ,3.9,4.0, 4.1,4.2, 4.3},
            legend pos= north west,
            ymajorgrids=true,
            grid style=dashed,
        ]
        \addplot[
            sharp plot,
            color=blue,
            mark=*,
        ]
        coordinates {
(1,4.16889837898783)
(2,4.12817352283799)
(4,4.10909024909956)
(6,4.0807928955174)
(8,4.06176064181311)
(10,4.07538327021504)
        };
        \addplot[
            sharp plot,
            color=orange,
            mark=*,
        ]
        coordinates {
(1,3.97091356109936)
(2,3.83279521760925)
(4,3.76352174949516)
(6,3.74439539734152)
(8,3.70886089217209)
(10,3.72005244530994)
        };
        \legend{Base, Opt}
        \end{axis}
        \end{tikzpicture}
        \end{adjustbox}
        \caption{Frequency - Total Driving Distance}
    \end{subfigure}
    \vspace{1cm}
    
    \begin{subfigure}[b]{0.49\textwidth}
        \centering
        \begin{adjustbox}{max width=\textwidth}
        \begin{tikzpicture}
        \begin{axis}[
            xlabel={Freq(/h)},
            ylabel={Tax},
            xmin=1, xmax=10,
            ymin=0.2, ymax=0.72,
            xtick={1, 2, 4, 6, 8, 10},
            ytick={0.2,0.4,0.6},
            legend pos=north west,
            ymajorgrids=true,
            grid style=dashed,
        ]
        \addplot[
            color=blue,
            mark=*,
            sharp plot,
            ]
            coordinates {
(1,0.27019273516501)
(2,0.544932555500356)
(4,0.644367011833324)
(6,0.665715002323926)
(8,0.695114567366673)
(10,0.701814660900382)
            };
        \end{axis}
        \end{tikzpicture}
        \end{adjustbox}
        \caption{Frequency - Tax}
    \end{subfigure}    
    \hfill
    \begin{subfigure}[b]{0.49\textwidth}
        \centering
        \begin{adjustbox}{max width=\textwidth}
        \begin{tikzpicture}
        \begin{axis}[
            xlabel={Freq(/h)},
            ylabel={Total Cost of the System (x1000)},
            xmin=1, xmax=10,
            ymin=1.05, ymax=1.7,
            xtick={1, 2, 4, 6, 8, 10},
            ytick={1.1, 1.2, 1.3, 1.4, 1.5, 1.6},
            legend pos=north west,
            ymajorgrids=true,
            grid style=dashed,
        ]
        \addplot[
            sharp plot,
            color=blue,
            mark=*,
        ]
        coordinates {
(1,1.07729221119594)
(2,1.09476600547525)
(4,1.10254768444355)
(6,1.10271719654919)
(8,1.10038995643552)
(10,1.10199239023388)
        };
        \addplot[
            sharp plot,
            color=orange,
            mark=*,
        ]
        coordinates {
(1,1.26084975275755)
(2,1.48025923631815)
(4,1.54691363793937)
(6,1.55944083498222)
(8,1.57166183620153)
(10,1.58269210248616)
        };
        \legend{Base, Opt}
        \end{axis}
        \end{tikzpicture}
        \end{adjustbox}
        \caption{Frequency - Total Cost of the System}
    \end{subfigure}
    \caption{Impact of train frequency per hour on the performances of base and optimal policies}
    \label{fig:SensitivityFreq}
\end{figure}

As shown in Figure \ref{fig:SensitivityFreq}(a), the optimal policy increases the modal shift from the base scenarios, starting from less than 10\% and reaching up to a maximum of 60\% with increasing scheduled line frequency. The modal shift appears to saturate after the frequency reaches 6, as higher frequencies do not further reduce the request waiting time at the station. Additionally, as the modal shift increases, the driving time under the optimal policy decreases with rising frequency. The relative driving time savings between the base and optimal policies also grow with increasing frequency. However, a higher tax is required with the increased modal shift, resulting in higher total system costs and increasing the freight forwarder's cost burden. Therefore, we may conclude that improving the services without any strategic changes can increase the modal shift to a certain extent.

\subsection{A trade-off between minimizing distances and minimizing total cost of transportation under different policies}

In this section, we explore various policies under different budget levels. We analyze how total distances and freight forwarder costs change with varying subsidy and tax levels to understand the impact on stakeholders. The subsidy levels range from 0 to 1 in increments of 0.1, and we consider three budget levels for the authority, represented as a percentage (0\%, 25\%, and 50\%) of $f^\text{full}$, the required budget to subsidize the schedule line without tax fully.  We apply the algorithm from Section \ref{sec:Algorithm} to determine the solutions for each scenario.

To illustrate the trade-off effect, we select the instance with the highest modal shift, referred to as Inter-Diff-W. Figure \ref{fig:cost_distance} depicts the trade-off between stakeholders’ objectives. For each budget, each point depicts a different subsidy level with the corresponding tax level. Solutions are plotted as points on the graph and connected to approximate a Pareto front. To ensure solution quality, the best solution from a lower subsidy level is used as the initial solution for a higher subsidy level. Infeasible points that do not meet budget constraints have been omitted. Additionally, we include the point for the base scenario.

\begin{figure}[h!]
    \centering
    \begin{tikzpicture}
    \begin{axis}[
        xlabel={Cost of freight forwarder},
        ylabel={Distances},
        xmin=1500, xmax=2600,
        ymin=5300, ymax=6350,
        legend pos=north east,
        width=10cm,
        height=8cm
    ]
    \addplot[
        mark=square*,
        solid,
        thick
    ] coordinates {
        (1766.48, 6115.164)
        (1861.01, 6036.44052)
        (1905.99270, 5825.30971)
        (2059.22, 5694.64933)
        (2124.396567, 5611.576526)
        (2196.981839, 5503.542764)
        (2328.263773, 5485.415041)
        (2460.517804, 5446.57443)
        (2503.34522, 5382.245296)
        (2524.54982312471, 5367.31366206947)
    };
    \addlegendentry{0\%}

    \addplot[
        mark=*,
        dotted,
        thick
    ] coordinates {
        (1727.767801, 5660.095112)
        (1896.77538, 5548.985729)
        (2007.537873, 5529.951718)
        (2090.386624, 5449.445234)
        (2168.22691, 5412.027818)
        (2219.54982312471, 5367.31366206947)
    };
    \addlegendentry{25\%}

    \addplot[
        mark=triangle*,
        dashed,
        thick
    ] coordinates {
        (1656.378427, 5585.231457)
        (1751.545733, 5440.666683)
        (1848.666092, 5374.79033)
        (1914.54982312471, 5367.31366206947)
    };
    \addlegendentry{50\%}
    \addplot [black, mark = x, nodes near coords=Base,every node near coord/.style={anchor=+90}] coordinates {(1772.2, 6290.0)};
    \end{axis}
    \end{tikzpicture}
    \caption{Trade-off between transportation authority and freight forwarder for policies with partial subsidy and optimal policy with different budgets. }
    \label{fig:cost_distance}
\end{figure}
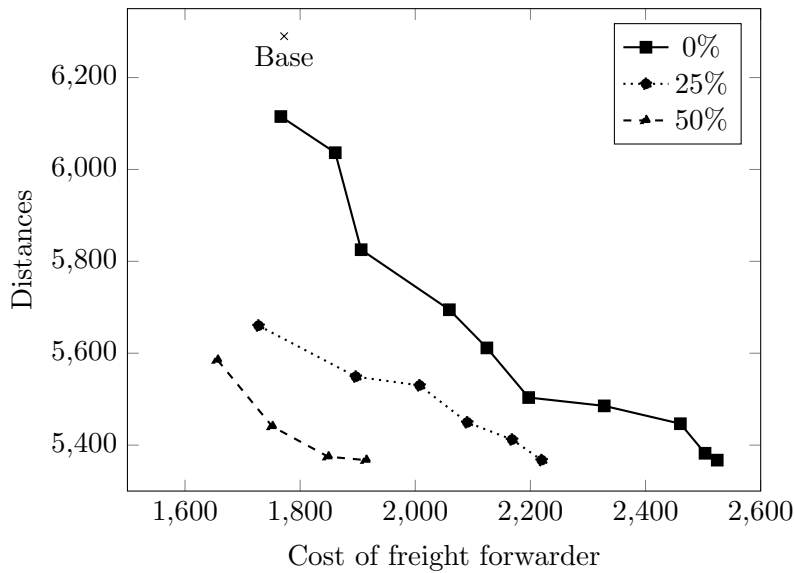

Figure \ref{fig:cost_distance} reveals two key observations. First, for each budget level, an increase in subsidy results in decreased distances but increased costs for freight forwarders. From lowest to optimal subsidy, the distance decreases from 3.9 to 5.1 \% while the cost of freight forwarder changes between 15.6 and 28.5\%.  The increased cost is due to the higher taxes required for greater subsidies. While the authority benefits most from higher subsidy levels, it must also consider the burden on freight forwarders. This finding supports Proposition \ref{prop:Proposition_2}, which states that increasing taxes while maintaining the budget level ensures no worse total driven distance.

Second, when the budget level increases while the subsidy level remains unchanged, distances tend to increase, and freight forwarder costs decrease. The distance increase is pretty small, with an increase of less than 1 \% on average, while the cost savings range from 13.3 to 28.9 \%. This is due to the reduced tax burden associated with a higher budget level. 

\section{Case Study}

In this section, we test the optimal policy on a case study in Berlin. We obtain order locations and corresponding distances from \cite{Sartori2020}. The locations were extracted from open addresses, and the corresponding distances were derived from the routing solution. The remaining parameters are generated as follows:

\begin{itemize}
    \item Data Sampling: We randomly sampled 100 requests from 5000 Berlin locations in a random instance, four depot locations (using delivery company depot locations and one imaginary location in the central area), and 40 vehicles.
    \item Scheduled Line Service Network: We selected a part of the S-Bahn Berlin light rail with 5 transfer nodes and 14 direct connections. The average headway is 10 minutes per train for all lines.
    \item The demand for each order is uniformly generated between 5 and 10. The time window is 1 hour. The vehicle capacity is 25 units, while the available train capacity is 60 units. We assume the vehicle speed is 60 km/h. The vehicle cost is 1.24 \texteuro/km, and the cost of transporting on the scheduled line service per demand is 2,4 \texteuro. 
\end{itemize}

Based on this setting, a scenario is shown in Figure \ref{fig:instance}. Black circles show the pickup and delivery locations, while the pink and red ones illustrate the depots and train stations.

\begin{figure}[htbp]
    \centering
    \includegraphics[width=0.75\textwidth]{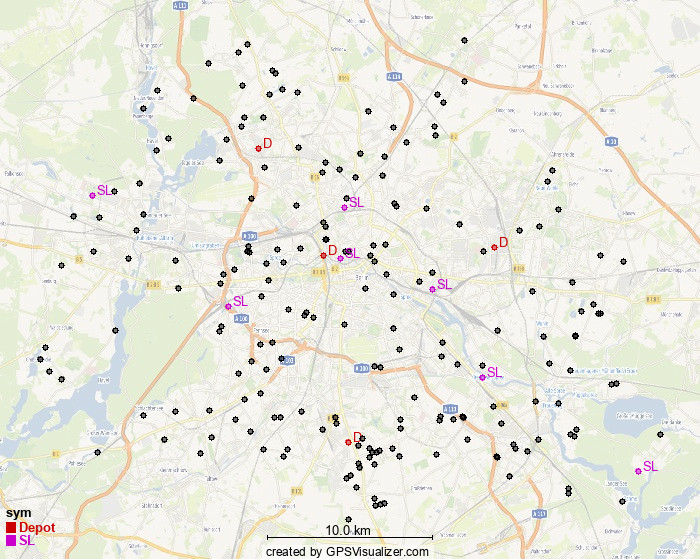}
    \caption{a Berlin case study}
    \label{fig:instance}
\end{figure}

\begin{table}[h]
\centering
\caption{Results of the Berlin Case Study where Base versus optimal (Opt) policy}
\label{tab:case_ber}
\scalebox{0.8}{
\begin{tabular}{@{}crrrrrrrrrrrrrc@{}}
\cmidrule(r){1-7} \cmidrule(l){9-15}
 & \multicolumn{3}{c}{Driving distance} & \multicolumn{1}{c}{} & \multicolumn{2}{c}{\%modal shift} & \multicolumn{1}{c}{} & \multicolumn{3}{c}{Operation cost} & \multicolumn{1}{c}{} & \multicolumn{2}{c}{\#Vehicles} &  \\ \cmidrule(lr){2-4} \cmidrule(lr){6-7} \cmidrule(lr){9-11} \cmidrule(lr){13-14}
Cost train & \multicolumn{1}{l}{Base} & \multicolumn{1}{l}{Opt} & \multicolumn{1}{l}{\%} & \multicolumn{1}{l}{} & \multicolumn{1}{l}{Base} & \multicolumn{1}{l}{Opt} & \multicolumn{1}{l}{} & \multicolumn{1}{l}{Base} & \multicolumn{1}{l}{Opt} & \multicolumn{1}{l}{\%} & \multicolumn{1}{l}{} & \multicolumn{1}{l}{Base} & \multicolumn{1}{l}{Opt} & tax \\ \cmidrule(r){1-4} \cmidrule(lr){6-7} \cmidrule(lr){9-11} \cmidrule(l){13-15} 
2 & 3376.5 & 3341 & -1.1 &  & 19.9 & 29.4 &  & 4446 & 4529 & 1.9 &  & 25.5 & 26.4 & \multicolumn{1}{r}{9.3} \\
4 & 3449.9 & 3349 & -2.9 &  & 4.0 & 23.2 &  & 4385 & 4766 & 8.7 &  & 24.4 & 25.7 & \multicolumn{1}{r}{14.8} \\ \bottomrule
\end{tabular}}
\end{table}

Table \ref{tab:case_ber} presents the average results of 10 scenarios comparing the base and optimal policies for the case study, with varying costs of transporting a demand unit on the scheduled line. The authority achieves greater driving distance savings when the cost of transporting on the scheduled line is higher. This is due to differences in the base scenarios. When the cost of using the scheduled line is low, the freight forwarder already achieves a relatively low driving distance, although not as low as the driven distance in the base scenario. The optimal policy provides a 1.1\% distance saving with a higher modal shift. Conversely, when the scheduled line cost is high, the freight forwarder uses it less often, resulting in a 4\% modal shift. Under the optimal policy, the modal shift increases to 23.2\%, and the driving distance saving is approximately 2.9\%. However, this saving comes with an increase in operational costs. This may prompt the authority to allocate an additional budget to alleviate the cost burden on the freight forwarder.

\section{Conclusions}
\label{sec:Conclusion}

This study investigates the impact of an integrated road tax and scheduled line service subsidy policy on promoting a freight modal shift from road to more sustainable modes for inner-city freight transportation. The subsidy is partly funded by road tax revenues to overcome political and economic challenges. We formulate a bilevel problem to capture the interaction between the transportation authority and freight forwarder. The upper level determines the policy, tax, and subsidy levels from an available budget, while the lower level involves freight forwarders’ routing decisions based on the given policy. We identify a condition for the optimal policy and derive its properties. Moreover, we show the benefits and consequences of optimal policies for each player. Since the PDPTW-SL scheme at the lower level depends on the tactical network decision, we also test the impact of train frequency and order scatteredness on stakeholders’ objectives. However, the optimal policy may burden the freight forwarder depending on the budget. We also show the trade-off between minimizing distance and total transportation costs under non-optimal policies. Finally, we validate our approach with a Berlin case study using open data from \cite{Sartori2020}.

We show that the optimal policy for the proposed problem is when the scheduled line services are fully subsidized. Moreover, under this policy, the budget does not influence the freight forwarder’s decision but can directly reduce their costs. In addition to the theoretical findings, we conducted extensive numerical tests. We show that the optimal policy can reduce the driving distance by up to 12.5\% and substantially increase the modal shift at higher operational costs from the tax level. Moreover, the scheme shows higher driving distance savings and modal shifts with higher train frequency up to a certain level. Furthermore, we numerically show that with an additional budget, the freight forwarder can save more costs than the budget with the partial subsidy policy at the cost of higher distance for the authority. For the case study, we can achieve driving distance savings and high modal shift with suitable cost settings.

Future research could focus on developing an exact algorithm to solve this bilevel program. Since the current work omits the initial cost of using more vehicles, we cannot address the trade-offs when the modal shift is higher. Incorporating vehicle costs can tackle this. Investigating scenarios with multiple freight forwarders and evaluating the impact of policies on different freight forwarders would also provide more practical insights. We can see which freight forwarder would feel the impact of the policy and to what extent.

\section*{CRediT author statement}

\textbf{Tundulyasaree K.}: Conceptualization, Methodology, Software, Validation, Formal analysis, Writing - Original Draft, Visualization. \textbf{Martin L.}: Conceptualization, Methodology, Formal analysis, Writing - Review \& Editing, Supervision. \textbf{Van Lieshout R. N.}: Conceptualization, Methodology, Formal analysis, Writing - Review \& Editing, Supervision. \textbf{Van Woensel T.}: Conceptualization, Writing - Review \& Editing, Supervision.

\section*{Acknowledgments}%
The first author expresses profound gratitude to the Royal Thai Government for providing the PhD scholarship at Eindhoven University of Technology.
The second and fourth authors thank the European Union for the support via the Interreg North-Sea Region project InnoWaTr. This work used the Dutch national e-infrastructure with the support of the SURF Cooperative using NWO/EINF-grant [EINF-9143].


\bibliographystyle{elsarticle-harv} 
\bibliography{cas-refs}

\appendix
\section{Proofs}
\label{app:proofs}

\begin{lemma}\label{lem:tradeoff}
The freight forwarder's driven distance decreases if and only if the amount on the scheduled line increases, i.e. for two feasible solutions $\langle s_1,t_1,d_1^\star,f_1^\star\rangle$, $\langle s_2, t_2,d_2^\star,f_2^\star\rangle$, it holds that $d_1^*<d_2^*\iff f_1^*>f_2^*$. 
\end{lemma}
\begin{proof}
    This immediately follows from the optimality of the freight forwarders decisions. If both $d_1^*<d_2^*$ and $f_1^*<f_2^*$, the second solution is dominated by the first, and can therefore never be a solution to the lower level problem. 
\end{proof}

\setcounter{proposition}{0}

\begin{proposition}\label{prop:Proposition_1}

For a given budget $B$, the freight forwarder's total cost decreases if and only if the driven distance is increasing, i.e., for two feasible solutions $\langle s_1,t_1,d_1^\star,f_1^\star\rangle$, $\langle s_2, t_2,d_2^\star,f_2^\star\rangle$, 
\begin{align}\label{eq:costFreightForwarder}
d_1^\star<d_2^* \iff \left(1 + t_1\right) d_1^\star + \left(1 - s_1\right) f_1^\star &> \left(1 + t_2\right) d_2^\star + \left(1 - s_2\right) f_2^\star 
\end{align}
\end{proposition}\
\begin{proof}

Let $\langle s_1,t_1,d_1^\star,f_1^\star\rangle$ and $\langle s_2,t_2,d_2^\star,f_2^\star\rangle$ be two feasible solutions for the problem with $0\leq s_1, s_2 \leq 1$ the subsidy in either solution, $0 \leq t_1, t_2$ the associated tax level, and $d_1^\star$ and $d_2^\star$ the driven distance in either solution and $f_1^\star$ and $f_2^\star$ the cost of using the scheduled line services. 
We start by rewriting the second equation in the proposition: 
\begin{align*}
    \left(1 + t_1\right) d_1^\star + \left(1 - s_1\right) f_1^\star &> \left(1 + t_2\right) d_2^\star + \left(1 - s_2\right) f_2^\star \\
    \iff d_1^\star +f_1^\star - t_1d_1^*+s_1f_1^* &> d_2^\star +  f_2^\star-t_2d_2^*+s_2f_2^*. 
\end{align*}
Using that $B=s_1f_1^\star - t_1d_1^\star= s_2 f_2^\star - t_2 d_2^\star$, we can further rewrite as 
\begin{align*}
    d_1^\star +f_1^\star-B &>d_2^\star +  f_2^\star-B \\
    \iff d_1^\star +f_1^\star &> d_2^\star +f_2^\star.
\end{align*}
Therefore, we continue by proving the following equivalence: 
$$d_1^\star<d_2^* \iff d_1^\star +f_1^\star > d_2^\star +f_2^\star.$$

\noindent ($\implies$) Since the freight forwarder makes optimal decisions under both policies, it holds that
\begin{align*}
    (1+t_2)d_2^*+(1-s_2)f_2^*& \leq (1+t_2)d_1^*+(1-s_2)f_1^* \\
    \iff d_2^*+f_2^*-d_1^*-f_1^*&\leq t_2(d_1^*-d_2*)+s_2(f_2^*-f_1^*).
\end{align*}
Given that $d_1^\star < d_2^\star$, by Lemma~\ref{lem:tradeoff} it holds that $f_1^\star > f_2^\star$, so the right hand side of the above equation is negative. Thus
\begin{align*}
    d_2^*+f_2^*-d_1^*-f_1^*&< 0 \\
    \iff d_1^*+f_1^* &> d_2^*+f_2^*.
\end{align*}

\noindent ($\Longleftarrow$) We now start with the optimality of the first lower level solution:
\begin{align*}
    (1+t_1)d_1^*+(1-s_1)f_1^*& \leq (1+t_1)d_2^*+(1-s_1)f_2^* \\
    \iff d_1^*+f_1^*-d_2^*-f_2^*&\leq t_1(d_2^*-d_1^*)+s_1(f_1^*-f_2^*).
\end{align*}
Given that $d_1^\star +f_1^\star > d_2^\star +f_2^\star$, the left hand side of the above equation must be positive. Therefore, at least one of the terms on the right hand side should be positive. By Lemma~\ref{lem:tradeoff}, both terms have the same parity, and must therefore both be positive. We conclude that $d_2^*>d_1^*$.
\end{proof}


\begin{proposition}
\label{prop:Proposition_2}
For a given budget $B$, the optimal travel distance $d^*$ decreases in the tax level $t$, i.e., for two feasible solutions $\langle s_1,t_1,d_1^\star,f_1^\star\rangle$, $\langle s_2,t_2,d_2^\star,f_2^\star\rangle$ with $t_1 < t_2$, it holds that 
$d_1^\star \geq d_2^\star$. 
\end{proposition}
\begin{proof}
    We distinguish two cases for the subsidy level. 

    Case I: $s_1\geq s_2$. It follows from the lower level optimality of the first solution that 
    \begin{align*}
        (1+t_1)d_1^*+(1-s_1)f_1^*&\leq (1+t_1)d_2^*+(1-s_1)f_2^* \\
         &\stackrel{t_1<t_2}{<} (1+t_2)d_2^*+(1-s_1)f_2^* \\
          &\stackrel{s_1\geq s_2}{<} (1+t_2)d_2^*+(1-s_2)f_2^*.
    \end{align*}
Applying Proposition~\ref{prop:Proposition_1}, we find that $d_1^*>d_2^*$.

Case II: $s_1<s_2$. We use proof by contradiction and assume that $d_1^*<d_2^*$. We again start from the lower level optimality of the first solution: 
\begin{align}
\left(1+t_1\right) d_1^\star + \left(1 - s_1\right) f_1^\star &\leq \left(1+t_1\right)d_2^\star + \left(1 - s_1\right) f_2^\star \nonumber\\
\iff \left(1+t_1\right) \left(d_1^\star - d_1^\star\right) &\geq \left(1-s_1\right)\left(f_1^\star - f_2^\star\right) \nonumber \\
\left(1+t_1\right) \left(d_2^\star - d_1^\star\right) &\stackrel{s_1< s_2}{>} \left(1-s_2\right)\left(f_1^\star - f_2^\star\right). \label{eq:Lemma1_eq1}\\
\intertext{Analogously, the second solution is optimal given subsidy rate $s_2$ and tax level $t_2$: }
\left(1+t_2\right) d_2^\star + \left(1 - s_2\right) f_2^\star &\leq \left(1+t_2\right)d_1^\star + \left(1 - s_2\right) f_1^\star \nonumber \\ 
\iff \left(1+t_2\right) \left(d_2^\star - d_1^\star\right) &\leq \left(1-s_2\right)\left(f_1^\star - f_2^\star\right) \label{eq:Lemma1_eq2}
\intertext{which with \eqref{eq:Lemma1_eq1} becomes}
\left(1+t_1\right) \left(d_2^\star - d_1^\star\right) &> \left(1+t_2\right)\left(d_2^\star - d_1^\star\right), \nonumber \label{eq:Lemma1_eq13} \\
\stackrel{t_1<t_2}{\iff}  d_2^*-d_1^*<0
\end{align}
leading to a contradiction with the assumption  $d_1^\star < d_2^\star$. 

\end{proof}

\begin{proposition}
\label{prop:newProp3}
For a given budget $B$ and subsidy level $s$, let $d^*(s)$ denote the lowest driven distance over all tax rates $t$. Then, $d^*(s)$ decreases monotonically in $s$.  
\end{proposition}
\begin{proof}
   Let $\langle s,t^*,d^\star(s),f^\star(s)\rangle$ denote the solution with the lowest driven distance given subsidy level $s$, and let $\langle s,t_2,d_2^\star,f_2^\star\rangle$ denote any other feasible solution with the same subsidy level, so $d_2^*>d^\star(s)$. Now consider a higher subsidy level $s'>s$ and new tax level $t'=(s'f^\star(s)-B)/d^\star(s)>t^*$. By the optimality of the first solution, we have that
   \begin{align*}
       (1+t')d^\star(s)+(1-s')f^\star(s) &=d^\star(s)+s'f^\star(s)-B+f^\star(s)-s'f^\star(s) \\
       &=d^\star(s)+sf^\star(s)-B+f^\star(s)-sf^\star(s) \\
       &=(1+t^*)d^\star(s)+(1-s)f^\star(s)\\
       &\leq (1+t^*)d_2^*+(1-s)f_2^* \\
       & = d_2^*+s(\frac{d_2^*}{d^*(s)}f^*(s)-f_2^*)+f_2-\frac{d_2^*}{d^*(s)}B \\
       &\leq d_2^*+s'(\frac{d_2^*}{d^*(s)}f^*(s)-f_2^*)+f_2-\frac{d_2^*}{d^*(s)}B \\
       &= d_2^*+\frac{t'd^\star(s)+B}{f^\star(s)}\frac{d_2^*}{d^*(s)}f^*(s)-s'f_2^*+f_2-\frac{d_2^*}{d^*(s)}B \\
       &=  (1+t')d^\star_2+(1-s')f^\star_2. 
   \end{align*}
In other words, there exists a tax rate $t'$ such that solution $\langle s',t',d^\star(s),f^\star(s)\rangle$ is preferred by the freight forwarder over all alternative solutions with a higher driven distance. Therefore, the minimum driven distance at the increased subsidy level $s'$ is at most $d^\star(s)$, i.e. $d^\star(s)\geq d^\star(s')$. 
\end{proof}

\begin{proposition}
\label{prop:newProp4}
Let $f^\text{full}$ denote the scheduled line costs of the freight forwarder under the policy $(s=1,t=0)$. If $B\leq f^\text{full}$, there exists an optimal where $s=1$. If $B>f^\text{full}$, the problem is infeasible. 
\end{proposition}
\begin{proof}
    Let $d^\text{full}$ denote the driven distance under policy $(s=1,t=0)$. Under this policy, the freight forwarder only minimizes distance, so $d^\text{full}$ is a lower bound on the optimal distance. 

    If $B\leq f^\text{full}$, the transportation authority can fully subsidize the scheduled line and set a tax rate of $t=(f^\text{full}-B)/d^\text{full}\geq 0$, which results in a driven distance of $d^\text{full}$, attaining the lower bound. 

     Now assume that $B>f^\text{full}$ and that there exists a feasible solution $(s,t,d^*,f^*)$. By the budget constraint, we have that $B=sf^*-td^*\leq sf^*\leq f^*$, contradicting the assumption that $B>f^\text{full}$. 
\end{proof}

\begin{proposition}    
\label{prop:Proposition_4}
Under the optimal policy, increasing the budget does not influence the freight forwarder's optimal routing decision but decreases its total cost, i.e., for any two optimal policies with $s_1 = s_2 = 1$ with their corresponding budget $B_1 < B_2 \leq f^\text{full}$ and total freight forwarder cost $C_1, C_2$, it holds that
\begin{align*}
    B_1 + C_1 = B_2 + C_2
\end{align*}
\end{proposition}
\begin{proof}
Let $\langle s_1,t_1,d_1^\star,f_1^\star\rangle$ and $\langle s_2,t_2,d_2^\star,f_2^\star\rangle$ be two optimal solutions for the policy maker's problem 
with $s_1=s_2$ the subsidy in either solution,
$t_1, t_2$ the associated tax level, $d_1^\star$ and $d_2^\star$ the distance traveled in either solution and $f_1^\star$ and $f_2^\star$ the total cost of putting freight on the scheduled line, such that $B_1 < B_2$ are their corresponding given budgets.

Provided that $s_1=s_2=1$, the freight forwarder's problem reduces to minimizing driven distance, resulting in $d_1^*=d_2^*$ and $f_1^*=f_2^*$, and therefore
\begin{align*}
d_1^\star + f_1^\star  &= d_2^\star + f_2^\star.  \\
\intertext{Subtracting and adding the corresponding budgets \eqref{eq:constr_budget} on both sides yields}
\underbrace{d_1^\star + f_1^\star - (s_1 f_1^\star - t_1 d_1^\star)}_{C_1} + B_1 &= \underbrace{d_2^\star + f_2^\star - (s_2 f_2^\star - t_2 d_2^\star)}_{C_2} + B_2 
\end{align*}
which is equivalent to the original statement. 
\end{proof}






\end{document}